\documentclass[draft]{amsart}
\usepackage{xcolor}

\usepackage[pdfborder={0 0 0},draft=false,pagebackref]{hyperref}

\usepackage{esint,amssymb,bm,tikz}

\newtheorem{thm}[equation]{Theorem}
\newtheorem{lem}[equation]{Lemma}

\theoremstyle{definition}

\theoremstyle{remark}
\newtheorem{rmk}[equation]{Remark}

\numberwithin{equation}{section}
\numberwithin{figure}{section}

\newcommand\abs[2][empty]{\csname#1\endcsname \lvert{#2}\csname#1\endcsname\rvert}
\newcommand\doublebar[2][empty]{\csname#1\endcsname \lVert{#2}\csname#1\endcsname\rVert}

\newcommand\mat[1]{\bm{#1}}
\newcommand\arr[1]{\dot{\bm{#1}}}

\newcommand\Div{\mathop{\mathrm{div}}\nolimits}
\newcommand\Tr{\mathop{\smash{\arr{\mathrm{Tr}}}\vphantom{T}}\nolimits}
\newcommand\WTr{\mathop{\smash{\widehat{\bm{\mathrm{Tr}}}} \vphantom{T}}\nolimits}
\newcommand\Trace{\mathop{\mathrm{Tr}}\nolimits}
\newcommand\M{\mathop{\smash{\arr{\mathrm{M}}}\vphantom{M}}\nolimits}
\newcommand\MM{\mathop{\smash{\widehat{\bm{\mathrm{M}}}}\vphantom{M}}\nolimits}

\newcommand\R{\mathbb{R}} 
\newcommand\C{\mathbb{C}} 
 
\newcommand\N{\mathbb{N}}
\newcommand\1{\mathbf{1}}
\newcommand\D{\mathcal{D}}
\newcommand\s{\mathcal{S}}
\newcommand\HH{\mathfrak{H}} 
\newcommand\CC{\mathfrak{C}} 
\newcommand\B{\mathfrak{B}} 
\newcommand\PP{\mathcal{N}} 

\newcommand\DD{\mathfrak{D}}
\newcommand\NN{\mathfrak{N}}
\newcommand\XX{\mathfrak{X}}
\newcommand\UU{\mathfrak{U}}
\newcommand\VV{\mathfrak{W}}

\def\dmn{{n+1}}

\def\dmn{d}

\newcommand\condT{\hyperlink{ConditionT}{\condTname}}
\newcommand\condM{\hyperlink{ConditionM}{\condMname}}
\newcommand\condS{\hyperlink{ConditionS}{\condSname}}
\newcommand\condD{\hyperlink{ConditionD}{\condDname}}
\newcommand\condG{\hyperlink{ConditionG}{\condGname}}
\newcommand\condTname{\textbf{\textup{(T)}}}
\newcommand\condMname{\textbf{\textup{(M)}}}
\newcommand\condSname{\textbf{\textup{(S)}}}
\newcommand\condDname{\textbf{\textup{(D)}}}
\newcommand\condGname{\textbf{\textup{(G)}}}
\newcommand\condTTname{\textbf{\textup{(T')}}}
\newcommand\condMMname{\textbf{\textup{(M')}}}
\newcommand\condSSname{\textbf{\textup{(S')}}}
\newcommand\condDDname{\textbf{\textup{(D')}}}
\newcommand\condGGname{\textbf{\textup{(G')}}}
\newcommand\condJSctsname{\textbf{\textup{(J1)}}}
\newcommand\condJDctsname{\textbf{\textup{(J2)}}}
\newcommand\condJSjumpname{\textbf{\textup{(J3)}}}
\newcommand\condJDjumpname{\textbf{\textup{(J4)}}}
\newcommand\condTT{\hyperlink{ConditionTT}{\condTTname}}
\newcommand\condMM{\hyperlink{ConditionMM}{\condMMname}}
\newcommand\condSS{\hyperlink{ConditionSS}{\condSSname}}
\newcommand\condDD{\hyperlink{ConditionDD}{\condDDname}}
\newcommand\condGG{\hyperlink{ConditionGG}{\condGGname}}
\newcommand\condJScts{\hyperlink{ConditionJScts}{\condJSctsname}}
\newcommand\condJDcts{\hyperlink{ConditionJDcts}{\condJDctsname}}
\newcommand\condJSjump{\hyperlink{ConditionJSjump}{\condJSjumpname}}
\newcommand\condJDjump{\hyperlink{ConditionJDjump}{\condJDjumpname}}

\usepackage{tikz}

\begin{document}

\title{Layer potentials for general linear elliptic systems}

\author{Ariel Barton}
\address{Ariel Barton, Department of Mathematical Sciences,
			309 SCEN,
			University of Ar\-kan\-sas,
			Fayetteville, AR 72701}
\email{aeb019@uark.edu}


\begin{abstract} 
In this paper we construct layer potentials for elliptic differential operators using the Lax-Milgram theorem, without recourse to the fundamental solution; this allows layer potentials to be constructed in very general settings. We then generalize several well known properties of layer potentials for harmonic and second order equations, in particular the Green's formula, jump relations, adjoint relations, and  Verchota's equivalence between well-posedness of boundary value problems and invertibility of layer potentials.
\end{abstract}


\maketitle

\tableofcontents


\section{Introduction}

There is by now a very rich theory of boundary value problems for Laplace's operator, and more generally for second order divergence form operators $-\Div \mat A\nabla$. The Dirichlet problem
\begin{equation*}-\Div \mat A\nabla u=0 \text{ in }\Omega,\quad u=f \text{ on }\partial\Omega, \quad \doublebar{u}_\XX\leq C\doublebar{f}_\DD\end{equation*}
and the Neumann problem
\begin{equation*}-\Div \mat A\nabla u=0 \text{ in }\Omega,\quad \nu\cdot\mat A\nabla u=g \text{ on }\partial\Omega, \quad \doublebar{u}_\XX\leq C\doublebar{g}_\NN\end{equation*}
are known to be well-posed for many classes of coefficients $\mat A$ and  domains~$\Omega$, and with solutions in many spaces $\XX$ and boundary data in many boundary spaces $\DD$ and~$\NN$.

A great deal of current research consist in extending these well posedness results to more general situations, such as operators of order $2m\geq 4$ (for example, \cite{MazMS10,KilS11B,MitMW11, MitM13A, BreMMM14, BarHM17pC}; see also the survey paper \cite{BarM16B}), operators with lower order terms (for example, \cite{BraBHV12,Tao12,Fel16,PanT16,DavHM16p}) and operators acting on functions defined on manifolds (for example, \cite{MitMT01,MitMS06,KohPW13}).

Two very useful tools in the second order theory are the double and single layer potentials given by 
\begin{align}
\label{eqn:introduction:D}
\D^{\mat A}_\Omega f(x) &= \int_{\partial\Omega} \overline{\nu\cdot \mat A^*(y)\nabla_{y} E^{L^*}(y,x)} \, f(y)\,d\sigma(y)
,\\
\label{eqn:introduction:S}
\s^L_\Omega g(x) &= \int_{\partial\Omega}\overline{E^{L^*}(y,x)} \, g(y)\,d\sigma(y)
\end{align}
where $\nu$ is the unit outward normal to~$\Omega$ and where $E^L(y,x)$ is the fundamental solution for the operator~$L=-\Div \mat A\nabla$, that is, the formal solution to $L E^L(\,\cdot\,,x)=\delta_x$. These operators are inspired by a formal integration by parts
\begin{align*}u(x) 
&= \int_\Omega \overline{L^*E^{L^*}(\,\cdot\,,x)}\,u
\\&=- \int_{\partial\Omega}\!\! \overline{\nu\cdot \mat A^*\nabla E^{L^*}(\,\cdot\,,x)} \, u\,d\sigma
+\int_{\partial\Omega}\!\!\overline{E^{L^*}(\,\cdot\,,x)} \, \nu\cdot \mat A\nabla u\,d\sigma
+\int_\Omega \overline{E^{L^*}(\,\cdot\,,x)}\,Lu\end{align*}
which gives the Green's formula
\begin{equation*}u(x) = -\D^{\mat A}_\Omega (u\big\vert_{\partial\Omega})(x) + \s^L_\Omega (\nu\cdot \mat A\nabla u)(x)\quad\text{if $x\in\Omega$ and $Lu=0$ in $\Omega$}\end{equation*}
at least for relatively well-behaved solutions~$u$.

Such potentials have many well known properties beyond the above Green's formula, including jump and adjoint relations. In particular, by a clever argument of Verchota \cite{Ver84} and some extensions in \cite{BarM13,BarM16A}, well posedness of the Dirichlet problem is equivalent to invertibility of the operator $g\mapsto \s^L_\Omega g\big\vert_{\partial\Omega}$, and well posedness of the Neumann problem is equivalent to invertibility of the operator $f\mapsto \nu\cdot\mat A\nabla\D^{\mat A}_\Omega f$. 

This equivalence has been used to solve boundary value problems in many papers, including 
\cite{FabJR78,Ver84,DahK87,FabMM98,Zan00} in the case of harmonic functions (that is, the case $\mat A=\mat I$ and $L=-\Delta$) and
\cite{AlfAAHK11, Bar13,Ros13, HofKMP15B, HofMayMou15,HofMitMor15, BarM16A} in the case of more general operators under various assumptions on the coefficients~$\mat A$. Layer potentials have been used in other ways in \cite{PipV92,KenR09,Rul07,Mit08,Agr09,MitM11,BarM13,AusM14}. Boundary value problems were studied using a functional calculus approach in \cite{AusAH08,AusAM10,AusA11, AusR12, AusM14, AusS14p, AusM14p}; in \cite{Ros13} it was shown that certain operators arising in this theory coincided with layer potentials.

Thus, it is desirable to extend layer potentials to more general situations. One may proceed as in the homogeneous second order case, by constructing the fundamental solution, formally integrating by parts, and showing that the resulting integral operators have appropriate properties. In the case of higher order operators with constant coefficients, this has been done in \cite{Agm57,CohG83, CohG85, Ver05, MitM13A, MitM13B}. However, all three steps are somewhat involved in the case of variable coefficient operators (although see \cite{DavHM16p,Bar16} for fundamental solutions, for second order operators with lower order terms, and for higher order operators without lower order terms, respectively).


An alternative, more abstract construction is possible. The fundamental solution for various operators was constructed in \cite{HofK07,Bar16,DavHM16p} as the kernel of the Newton potential, which may itself be constructed very simply using the Lax-Milgram theorem. It is possible to rewrite the formulas \eqref{eqn:introduction:D} and~\eqref{eqn:introduction:S} for the second order layer potential directly in terms of the Newton potential, without mediating by the fundamental solution, and this construction generalizes very easily. It is this approach that was taken in \cite{BarHM15p,BarHM17pA}. 

In this paper we will provide the details of this construction in a very general context. Roughly, this construction is valid for all differential operators $L$ that may be inverted via the Lax-Milgram theorem, and all domains $\Omega$ for which suitable boundary trace operators exist. We will also show that many properties of traditional layer potentials are valid in the general case.

The organization of this paper is as follows. The goal of this paper is to construct layer potentials associated to an operator~$L$ as bounded linear operators from a space $\DD_2$ or $\NN_2$ to a Hilbert space $\HH_2$ given certain conditions on $\DD_2$, $\NN_2$ and $\HH_2$.
In Section~\ref{sec:dfn} we will list these conditions and define our terminology. Because these properties are somewhat abstract, in
Section~\ref{sec:example} we will give an example of spaces $\HH_2$, $\DD_2$ and $\NN_2$ that satisfy these conditions in the case where $L$ is a higher order differential operator in divergence form without lower order terms. 

This is the context of the paper \cite{BarHM17pC}; we intend to apply the results of the present paper therein to solve the Neumann problem with boundary data in $L^2$ for operators with $t$-independent self-adjoint coefficients.

In Section~\ref{sec:D:S} of this paper we will provide the details of the construction of layer potentials.
We will prove the higher order analogues for the Green's formula, adjoint relations, and jump relations in Section~\ref{sec:properties}. 

Finally, in Section~\ref{sec:invertible} we will show that the equivalence between well posedness of boundary value problems and invertibility of layer potentials of \cite{Ver84,BarM13,BarM16A} extends to the higher order case. 

\section{Terminology}
\label{sec:dfn}

We will construct layer potentials $\D^\B_\Omega$ and $\s^L_\Omega$ using the following objects.

\begin{itemize}
\item Two Hilbert spaces $\HH_1$ and $\HH_2$. 
\item Six vector spaces $\widehat\HH_1^\Omega$, $\widehat\HH_1^\CC$, $\widehat\HH_2^\Omega$, $\widehat\HH_2^\CC$, $\widehat\DD_1$ and $\widehat\DD_2$.
\item Bounded bilinear functionals $\B:\HH_1\times\HH_2\mapsto \C$, $\B^\Omega:\HH_1^\Omega\times\HH_2^\Omega\mapsto \C$, and $\B^\CC:\HH_1^\CC\times\HH_2^\CC\mapsto \C$. (We will define the spaces $\HH_j^\Omega$, $\HH_j^\CC$ momentarily.)
\item Bounded linear operators $\Tr_1:\HH_1\mapsto\widehat\DD_1$ and $\Tr_2:\HH_2\mapsto\widehat\DD_2$.
\item Bounded linear operators from $\HH_j$ to $\widehat\HH_j^\Omega$ and $\widehat\HH_j^\CC$; we shall denote these operators $\big\vert_\Omega$ and~$\big\vert_\CC$. 
\end{itemize}

We will work not with the spaces $\widehat \HH_j^\Omega$, $\widehat\HH_j^\CC$ and $\widehat\DD_j$, but with the spaces 
$ \HH_j^\Omega$, $\HH_j^\CC$ and $\DD_j$ defined as follows.
\begin{gather}
\HH_j^\Omega=\{F\big\vert_\Omega:F\in\HH_j\}/\sim\text{ with norm }\doublebar{f}_{\HH^\Omega_j} = \inf\{\doublebar{F}_{\HH_j}: F\big\vert_\Omega=f\}
,\\
\HH_j^\CC=\{F\big\vert_\CC : F\in\HH_j\}/\sim\text{ with norm }\doublebar{f}_{\HH^\CC_j} = \inf\{\doublebar{F}_{\HH_j}: F\big\vert_\CC=f\}
,\\
\DD_j=\{\Tr_j F:F\in\HH_j\}/\sim\text{ with norm }\doublebar{f}_{\DD_j} = \inf\{\doublebar{F}_{\HH_j}: \Tr_j F=f\}
\end{gather}
where $\sim$ denotes the equivalence relation $f\sim g$ if $\doublebar{f-g}=0$.



We impose the following conditions on the given function spaces and operators. We require that there is some $\lambda>0$ such that for every $u\in\HH_1$, $v\in \HH_2$ and $\varphi$,~$\psi\in \HH_j$ for $j=1$ or $j=2$, the following conditions are valid.
\begin{gather}
\label{cond:coercive}
\sup_{w\in \HH_1\setminus\{0\}} \frac{\abs{\B(w,v)}}{\doublebar{w}_{\HH_1}}\geq \lambda \doublebar{v}_{\HH_2},\quad
	\sup_{w\in \HH_2\setminus\{0\}} \frac{\abs{\B(u,w)}}{\doublebar{w}_{\HH_2}}\geq \lambda \doublebar{u}_{\HH_1}.
\\
\label{cond:local}
\B(u,v) = \B^\Omega(u\big\vert_{\Omega},v\big\vert_\Omega) +\B^\CC(u\big\vert_{\CC}, v\big\vert_{\CC}). 
\\
\label{cond:trace:extension}
{\text{If $\Tr_j \varphi=\Tr_j \psi$, then there is a $w\in\HH_j$ with}}
\qquad\qquad \qquad\qquad\qquad \qquad 
\\\nonumber \qquad\qquad \qquad\qquad\qquad \qquad
{\text{
$w\big\vert_\Omega=\varphi\big\vert_\Omega$, $w\big\vert_{\CC}=\psi\big\vert_{\CC}$ and $\Tr_j w= \Tr_j\varphi=\Tr_j\psi$.}}
\end{gather}


We now introduce some further terminology.

We will define the linear operator $L$ as follows. If $u\in \HH_2$, let $Lu$ be the element of the dual space $\HH_1^*$ to $\HH_1$ given by
\begin{equation}\label{dfn:L}\langle\varphi,Lu\rangle = \B(\varphi,u).\end{equation}
Notice that $L$ is bounded $\HH_2\mapsto\HH_1^*$.



If $u\in \HH_2^\Omega$, we let $(Lu)\big\vert_\Omega$ be the element of the dual space to $\{\varphi\in\HH_1:\Tr_1 \varphi=0\}$ given by
\begin{equation}\label{dfn:L:interior}
\langle\varphi,(Lu)\big\vert_\Omega\rangle = \B^\Omega(\varphi\big\vert_\Omega,u)\quad\text{for all $\varphi\in\HH_1$ with $\Tr_1\varphi=0$}
.\end{equation}
If $u\in\HH_2$, we will often use $(Lu)\big\vert_\Omega$ as shorthand for $(L(u\big\vert_\Omega))\big\vert_\Omega$.
We will primarily be concerned with the case $(Lu)\big\vert_\Omega=0$. 

Let
\begin{equation}\NN_2=\DD_1^*, \qquad \NN_1=\DD_2^*\end{equation}
denote the dual spaces to $\DD_1$ and $\DD_2$.
We will now define the Neumann boundary values of an element $u$ of $\HH_2^\Omega$ that satisfies $(Lu)\big\vert_\Omega=0$.
If $\Tr_1\varphi=\Tr_1\psi$ and $(Lu)\big\vert_\Omega=0$, then $\B^\Omega(\varphi\big\vert_\Omega-\psi\big\vert_\Omega,u)=0$ by definition of $(Lu)\big\vert_\Omega$. Thus, $\B^\Omega(\varphi\big\vert_\Omega,u)$ depends only on~$\Tr_1\varphi$, not on~$\varphi$. Thus,  $\M_\Omega^\B u$ defined as follows is a well defined element of $\NN_2$.
\begin{equation}\label{eqn:Neumann}\langle \Tr_1\varphi,\M_\Omega^\B u \rangle = \B^\Omega(\varphi\big\vert_\Omega,u)\quad\text{for all $\varphi\in \HH_1$}.\end{equation}

We may compute
\begin{equation*}\abs{\langle \arr f,\M_\Omega^\B u \rangle} \leq \doublebar{\B^\Omega} \inf\{\doublebar{\varphi}_{\HH_1}:\Tr_1\varphi=\arr f\} \doublebar{u}_{\HH_2^\Omega}
=
\doublebar{\B^\Omega} \doublebar{\arr f}_{\DD_1}\doublebar{u}_{\HH_2^\Omega}\end{equation*}
and so we have the bound
$\doublebar{\M_\Omega^\B u }_{\NN_2}\leq \doublebar{\B^\Omega}\doublebar{u}_{\HH_2^\Omega}$.

If $(Lu)\big\vert_\Omega\neq 0$, then the linear operator given by $\varphi\mapsto B^\Omega(\varphi\big\vert_\Omega,u)$ is still of interest. We will denote this operator $L(u\1_\Omega)$; that is,
if $u\in\HH_2^\Omega$ (or $u\in\HH_2$ as before), then $L(u\1_\Omega)\in \HH_1^*$ is defined by
\begin{equation}\label{dfn:L:singular}
\langle\varphi,L(u\1_\Omega)\rangle = \B^\Omega(\varphi\big\vert_\Omega,u)\quad\text{for all $\varphi\in\HH_1$}
.\end{equation}

\section{An example: higher order differential equations}
\label{sec:example}

In this section, we provide an example of a situation in which the terminology of Section~\ref{sec:dfn} and the construction and properties of layer potentials of Sections~\ref{sec:D:S} and~\ref{sec:properties} may be applied. We remark that this is the situation of \cite{BarHM17pC}, and that we will therein apply the results of this paper.

Let $m\geq 1$ be an integer, and let $L$ be an elliptic differential operator of the form
\begin{equation}
\label{eqn:L}
Lu=(-1)^m\sum_{\abs\alpha=\abs\beta= m} \partial^\alpha(A_{\alpha\beta} \partial^\beta u)\end{equation}
for some bounded measurable coefficients~$\mat A$ defined on $\R^\dmn$. Here $\alpha$ and $\beta$ are multiindices in $\N_0^\dmn$, where $\N_0$ denotes the nonnegative integers. 
As is standard in the theory, we say that $Lu=0$ in an open set $\Omega$ in the weak sense if
\begin{equation}\label{eqn:weak}\int_\Omega \sum_{\abs\alpha=\abs\beta= m} \partial^\alpha\varphi \,A_{\alpha\beta}\, \partial^\beta u=0 \quad\text{for all $\varphi\in C^\infty_0(\Omega)$}.\end{equation}



We impose the following ellipticity condition: we require that for some $\lambda>0$,
\begin{equation*}\Re \sum_{{\abs\alpha=\abs\beta= m}} \int_{\R^\dmn} \overline{\partial^\alpha\varphi}\,A_{\alpha\beta}\,\partial^\beta\varphi\geq \lambda \doublebar{\nabla^m\varphi}_{L^2(\R^\dmn)}^2 \quad \text{for  all $\varphi\in\dot W^2_m(\R^\dmn)$.}
\end{equation*}

Let $\Omega\subset\R^\dmn$ be a Lipschitz domain with connected boundary, and let $\CC=\R^\dmn\setminus\bar\Omega$ denote the interior of its complement. Observe that $\partial\Omega=\partial\CC$. 

The following function spaces and linear operators satisfy the conditions of Section~\ref{sec:dfn}.
\begin{itemize}
\item $\HH_1=\HH_2=\HH$ is the homogeneous Sobolev space $\dot W^2_m(\R^\dmn)$ of locally integrable functions~$\varphi$ (or rather, of equivalence classes of functions modulo polynomials of degree $m-1$) with weak derivatives of order $m$, and such that the $\HH$-norm given by $\doublebar{\varphi}_\HH=\doublebar{\nabla^m\varphi}_{L^2(\R^\dmn)}$ is finite. This space is a Hilbert space with inner product $\langle \varphi,\psi\rangle =\sum_{\abs\alpha=m} \int_{\R^\dmn} \overline{\partial^\alpha\varphi}\,\partial^\alpha\psi$.
\item $\widehat\HH^\Omega$ and $\widehat\HH^\CC$ are the Sobolev spaces $\widehat\HH^\Omega=\dot W^2_m(\Omega)=\{\varphi:\nabla^m\varphi\in L^2(\Omega)\}$ and $\widehat\HH^\CC=\dot W^2_m(\CC)=\{\varphi:\nabla^m\varphi\in L^2(\CC)\}$ with the obvious norms. 

\item $\widehat \DD$ denotes the (vector-valued) Besov space $ \dot B^{2,2}_{1/2}(\partial\Omega)$ of locally integrable functions modulo constants with norm
\begin{equation*}\doublebar{f}_{\dot B^{2,2}_{1/2}(\partial\Omega)}
= \biggl(\int_{\partial\Omega}\int_{\partial\Omega} \frac{\abs{f(x)-f(y)}^2}{\abs{x-y}^\dmn}\,d\sigma(x)\,d\sigma(y)\biggr)^{1/2}.\end{equation*}

\item In \cite{Bar16pA,BarHM15p,BarHM17pC}, $\Tr$ is the linear operator defined on $ \HH$ by $\Tr u=\Trace^\Omega \nabla^{m-1}u\big\vert_\Omega$, where $\Trace^\Omega$ is the standard boundary trace operator of Sobolev spaces. (Given a suitable modification of the trace space $\DD$, it is also possible to choose 
$\Tr u = \{\Trace^\Omega \partial^\gamma u\}_{\abs\gamma\leq m-1}$, or more concisely $\Tr u = (\Trace^\Omega u,\partial_\nu  u,\dots,\partial_\nu^{m-1} u)$,  where $\nu$ is the unit outward normal, so that the boundary derivatives of $u$ of all orders are recorded. See, for example, 
\cite{
PipV95B,She06B,
Agr07,
MazMS10,MitM13A}.)
\item $\B$ is the bilinear operator on $\HH\times\HH$ given by 
\begin{equation*}\B(\psi,\varphi) = \sum_{\abs\alpha=\abs\beta= m}\int_{\R^\dmn}  \overline{\partial^\alpha\psi}\,A_{\alpha\beta}\,\partial^\beta\varphi.\end{equation*}
$\B^\Omega$ and $\B^\CC$ are defined analogously to~$\B$, but with the integral over $\R^\dmn$ replaced by an integral over $\Omega$ or~$\CC$.
\end{itemize}

$\B$, $\B^\Omega$ and $\B^\CC$ are clearly bounded and bilinear, and the restriction operators $\big\vert_\Omega:\HH\mapsto\widehat\HH^\Omega$,
$\big\vert_\CC:\HH\mapsto\widehat\HH^\CC$ are bounded and linear. 

The trace operator $\Tr$ is linear. If $\Omega=\R^\dmn_+$ is the half-space, then boundedness of $\Tr:\HH\mapsto\DD$ was established in \cite[Section~5]{Jaw77}; this extends to the case where $\Omega$ is the domain above a Lipschitz graph via a change of variables. If $\Omega$ is a bounded Lipschitz domain, then boundedness of $\Tr:W\mapsto\widehat\DD$, where $W$ is the inhomogeneous Sobolev space with norm $\sum_{k=0}^m \doublebar{\nabla^k\varphi}_{L^2(\R^\dmn)}$, was established in \cite[Chapter~V]{JonW84}. Then boundedness of $\Tr:\HH\mapsto\widehat\DD$ follows by the Poincar\'e inequality.

By assumption, the coercivity condition~\eqref{cond:coercive} is valid. If $\partial\Omega$ has Lebesgue measure zero, then Condition~\eqref{cond:local} is valid. A straightforward density argument shows that if $\Tr$ is bounded, then Condition~\eqref{cond:trace:extension} is valid.


Thus, the given spaces and operators satisfy the conditions imposed at the beginning of Section~\ref{sec:dfn}.

We now comment on a few of the other quantities defined in Section~\ref{sec:dfn}.

If $u\in \HH$, and if $Lu=0$ in $\Omega$ in the weak sense of formula~\eqref{eqn:weak}, then by density $\B^\Omega(\varphi,u)=0$ for all $\varphi\in\HH$ with $\Tr\varphi=0$; that is, $(Lu)\big\vert_\Omega$ as defined in Section~\ref{sec:dfn} satisfies $(Lu)\big\vert_\Omega=0$.

For many classes of domains there is a bounded extension operator from $\widehat\HH^\Omega$ to $\HH$, and so $\HH^\Omega=\widehat\HH^\Omega=\dot W^2_m(\Omega)$ with equivalent norms. (If $\Omega$ is a Lipschitz domain then this is a well known result of Calder\'on \cite{Cal61} and Stein \cite[Theorem~5, p.~181]{Ste70}; the result is true for more general domains, see for example \cite{Jon81}.)

As mentioned above, if $\Omega\subset\R^\dmn$ is a Lipschitz domain, then $\Tr$ is a bounded operator $\HH \mapsto \widehat \DD$. If $\partial\Omega$ is connected, then $\Tr$ moreover has a bounded right inverse. (See \cite{JonW84} or \cite[Proposition~7.3]{MazMS10} in the inhomogeneous case, and \cite{Bar16pB} in the present homogeneous case.) Thus, the norm in $\DD$ is comparable to the Besov norm. 
Furthermore, $\{\nabla^{m-1}\varphi\big\vert_{\partial\Omega}:\varphi\in C^\infty_0(\R^\dmn)\}$ is dense in $\DD$. Thus, if $m=1$ then $\DD=\widehat\DD=\dot B^{2,2}_{1/2}(\partial\Omega)$. If $m\geq 2$ then $\DD$ is a closed {proper} subspace of $\widehat\DD$, as the different partial derivatives of a common function must satisfy certain compatibility conditions. In this case $\DD$ is the Whitney-Sobolev space 
used in many papers, including \cite{AdoP98, 
MazMS10, MitMW11, MitM13A, MitM13B, BreMMM14, Bar16pA}.

If $m=1$, then by an integration by parts argument we have that $\M_\B^\Omega u = \nu\cdot \mat A\nabla u$, where $\nu$ is the unit outward normal to~$\Omega$, whenever $u$ is sufficiently smooth. The weak formulation of Neumann boundary values of formula~\eqref{eqn:Neumann} coincides with the formulation of higher order Neumann boundary data of \cite{BarHM15p,Bar16pA,BarHM17pC} given the above choice of $\Tr=\Trace^\Omega \nabla^{m-1}$, and with that of \cite{
Ver05,Agr07,
MitM13A} if we instead choose $\Tr u=(\Trace^\Omega u, \partial_\nu u,\dots,\partial_\nu^{m-1} u)$ or $\Tr u = \{\Trace^\Omega \partial^\gamma u\}_{\abs\gamma\leq m-1}$. 

%


\section{Construction of layer potentials}
\label{sec:D:S}

We will now use the Babu\v{s}ka-Lax-Milgram theorem to construct layer potentials. This theorem may be stated as follows.
\begin{thm}[{\cite[Theorem~2.1]{Bab70}}]
\label{thm:lax-milgram}
Let $\HH_1$ and $\HH_2$ be two Hilbert spaces, and let $\B$ be a bounded bilinear form on $\HH_1\times \HH_2$ that is coercive in the sense that for some fixed $\lambda>0$, formula~\eqref{cond:coercive} is valid for every $u\in\HH_1$ and $v\in\HH_2$.

Then for every linear functional $T$ defined on ${\HH_1}$ there is a unique $u_T\in {\HH_2}$ such that $\B(v,u_T)=\overline{T(v)}$. Furthermore, $\doublebar{u_T}_{\HH_2}\leq \frac{1}{\lambda}\doublebar{T}_{\HH_1\mapsto\C}$.
\end{thm}

We construct layer potentials as follows.

Let $\arr g\in\NN_2$. Then the operator $T_{\arr g}\varphi = \langle \arr g,\Tr_1\varphi\rangle$ is a bounded linear operator on~$\HH_1$. By the Lax-Milgram lemma, there is a unique $u_T=\s^L_\Omega\arr g\in \HH_2$ such that 
\begin{equation}\label{eqn:S}
\B( \varphi, \s^L_\Omega\arr g ) = \langle \Tr \varphi,\arr g\rangle
\quad\text{for all $\varphi\in\HH_1$}.\end{equation}
We will let $\s^L_\Omega\arr g$ denote the single layer potential of~$\arr g$. Observe that the dependence of $\s^L_\Omega$ on the parameter $\Omega$ consists entirely of the dependence of the trace operator on~$\Omega$, and the connection between $\Tr_2$ and $\Omega$ is given by formula~\eqref{cond:trace:extension}. This formula is symmetric about an interchange of $\Omega$ and $\CC$, and so $\s^L_\Omega \arr g=\s^L_\CC\arr g$.

The double layer potential is somewhat more involved. We begin by defining the Newton potential.

Let $H$ be an element of the dual space $\HH_1^*$ to~$\HH_1$. By the Lax-Milgram theorem, there is a unique element $\PP^L H$ of $\HH_2$ that satisfies
\begin{equation}
\label{eqn:newton}
\B(\varphi,\PP^L H) = \langle \varphi,H\rangle\quad\text{for all $\varphi\in\HH_1$}.\end{equation}
We refer to $\PP^L$ as the Newton potential.
In some applications, it is easier to work with the Newton potential rather than the single layer potential directly; we remark that
\begin{equation}\s^L_\Omega \arr g = \PP^L (T_{\arr g}) \quad\text{where } \langle T_{\arr g},\varphi \rangle = \langle \arr g,\Tr_1\varphi\rangle.\end{equation}

We now return to the double layer potential. Let $\arr f\in\DD_2$. Then there is some $F\in \HH_2$ such that $\Tr_2 F=\arr f$. Let
\begin{equation}
\label{eqn:D:+}\D_\Omega^\B \arr f = -F\big\vert_\Omega + \PP^L (L(\1_\Omega F))\big\vert_\Omega
\qquad\text{if $\Tr_2 F=\arr f$}.\end{equation}
Notice that $\D_\Omega^\B \arr f$ is an element of $\HH_2^\Omega$, not of $\HH_2$.

%

We will conclude this section by showing that $\D^\B_\Omega\arr f$ is well defined, that is, does not depend on the choice of $F$ in formula~\eqref{eqn:D:+}. We will also establish that layer potentials are bounded operators.

\begin{lem}\label{lem:potentials:bounded}
The double layer potential is well defined. Furthermore, we have the bounds
\begin{gather*}\doublebar{\D^\B_\Omega\arr f}_{\HH_2^\Omega} \leq \frac{\doublebar{B^\CC}}{\lambda}\doublebar{\arr f}_{\DD_2},
\quad
\doublebar{\D^\B_\CC\arr f}_{\HH_2^\CC} \leq \frac{\doublebar{B^\Omega}}{\lambda}\doublebar{\arr f}_{\DD_2},
\quad
\doublebar{\s^L_\Omega\arr g}_{\HH_2} \leq \frac{1}{\lambda}\doublebar{\arr g}_{\NN_2}.
\end{gather*}
\end{lem}

\begin{proof}
By Theorem~\ref{thm:lax-milgram}, we have that
\begin{equation*}\doublebar{\s^L_\Omega \arr g}_{\HH_2} 
\leq 
\frac{1}{\lambda}
\doublebar{T_{\arr g}}_{\HH_1\mapsto\C}
\leq 
\frac{1}{\lambda}
\doublebar{\Tr_1}_{\HH_1\mapsto\DD_1}\doublebar{\arr g}_{\DD_1\mapsto\C}.
\end{equation*}
By definition of $\DD_1$ and $\NN_2$, $\doublebar{\Tr_1}_{\HH_1\mapsto\DD_1}=1$ and $\doublebar{\arr g}_{\DD_1\mapsto\C}=\doublebar{\arr g}_{\NN_2}$, and so $\s^L_\Omega:\NN_2\mapsto\HH_2$ is bounded with operator norm at most $1/\lambda$.

We now turn to the double layer potential. We will begin with a few properties of the Newton potential.
By definition of~$L$, if $\varphi\in\HH_1$ then $\langle \varphi, LF\rangle = \B(\varphi, F)$. By definition of~$\PP^L$, $\B(\varphi,\PP^L (LF)) = \langle\varphi, LF\rangle$. Thus, by coercivity of~$\B$,
\begin{equation}F=\PP^L(LF)\quad\text{for all }F\in\HH_2.\end{equation}

By definition of $\B^\Omega$, $\B^\CC$ and $L(\1_\Omega F)$,
\begin{align*}
\langle\varphi,LF\rangle 
&= \B(\varphi,F)
=\B^\Omega(\varphi\big\vert_\Omega,F\big\vert_\Omega)
+\B^\CC(\varphi\big\vert_{\CC},F\big\vert_{\CC})&
\\&=
\langle\varphi, L(F\1_\Omega)\rangle + \langle \varphi, L(F\1_{\CC})\rangle
&\text{for all $\varphi\in\HH_1$}.\end{align*}
Thus, $LF=L(F\1_\Omega)+L(F\1_{\CC})$ and so
\begin{align}
\label{eqn:D:alternate:extensions}
-F + \PP^L (L(F\1_\Omega))
&=  -F + \PP^L (LF) - \PP^L (F\1_{\CC})
\\\nonumber&=  - \PP^L (L(F\1_{\CC}))
.\end{align}

In particular, suppose that $\arr f=\Tr_2 F=\Tr_2 F'$. By Condition~\eqref{cond:trace:extension}, there is some $F''\in \HH_2$ such that $F''\big\vert_\Omega=F\big\vert_\Omega$ and $F''\big\vert_{\CC}=F'\big\vert_{\CC}$.  Then
\begin{align*} -F\big\vert_\Omega + \PP^L (L(\1_\Omega F))\big\vert_\Omega
&= -F''\big\vert_\Omega + \PP^L (L(\1_\Omega F''))\big\vert_\Omega
=  - \PP^L (L(F''\1_{\CC}))\big\vert_\Omega
\\&=  - \PP^L (L(F'\1_{\CC}))\big\vert_\Omega
= -F'\big\vert_\Omega + \PP^L (L(\1_\Omega F'))\big\vert_\Omega
\end{align*}
and so $\D_\Omega^\B \arr f$ is well-defined, that is, depends only on $\arr f$ and not the choice of function $F$ with $\Tr_2 F=\arr f$. 

Furthermore, we have the alternative formula
\begin{equation}\label{eqn:D:alternate}
\D_\Omega^\B \arr f = - \PP^L (L(\1_{\CC} F))\big\vert_\Omega
\qquad\text{if $\Tr_2 F=\arr f$}
.\end{equation}

Thus,
\begin{equation*}\doublebar{\D_\Omega^B\arr f}_{\HH_2^\Omega}
\leq
\inf_{\Tr_2 F=\arr f} \doublebar{\PP^L (L(\1_{\CC} F))\big\vert_\Omega}_{\HH_2^\Omega}
\leq
\inf_{\Tr F=\arr f} \doublebar{\PP^L (L(\1_{\CC} F))}_{\HH_2}.
\end{equation*}
by definition of the $\HH_2^\Omega$-norm.

By Theorem~\ref{thm:lax-milgram} and definition of $\PP^L$, we have that
\begin{equation*}\doublebar{\PP^L (L(\1_{\CC} F))}_{\HH_2}\leq \frac{1}{\lambda} \doublebar{L(\1_{\CC} F)}_{\HH_1\mapsto\C}.\end{equation*}
Since $L(\1_{\CC} F)(\varphi) = \B^\CC(\varphi\big\vert_{\CC}, F\big\vert_{\CC})$, we have that \begin{equation*}\doublebar{L(\1_{\CC} F)}_{\HH_1\mapsto\C}\leq \doublebar{\B^\CC}\doublebar{F\big\vert_{\CC}}_{\HH_2^\CC}
\leq \doublebar{\B^\CC}\doublebar{F}_{\HH_2}\end{equation*}
and so
\begin{equation*}\doublebar{\D_\Omega^B\arr f\big\vert_\Omega}_{\HH_2^\Omega}
\leq
\inf_{\Tr F=\arr f} \frac{1}{\lambda} \doublebar{\B^\CC}\doublebar{F}_{\HH_2}
=\frac{1}{\lambda} \doublebar{\B^\CC} \doublebar{\arr f}_{\DD_2}
\end{equation*}
as desired.
\end{proof}


\section{Properties of layer potentials}
\label{sec:properties}

We will begin this section by showing that layer potentials are solutions to the equation $(Lu)\big\vert_\Omega=0$ (Lemma~\ref{lem:potentials:solutions}). We will then prove the Green's formula (Lemma~\ref{lem:green}), the adjoint formulas for layer potentials (Lemma~\ref{lem:adjoint}), and conclude this section by proving the jump relations for layer potentials (Lemma~\ref{lem:jump}).

\begin{lem}\label{lem:potentials:solutions} Let $\arr f\in\DD$, $\arr g\in\NN$, and let $u=\D^\B_\Omega\arr f$ or $u=\s^L_\Omega\arr g\big\vert_\Omega$. Then
$(Lu)\big\vert_\Omega=0$.
\end{lem}
\begin{proof}
Recall that $(Lu)\big\vert_\Omega=0$ if $\B^\Omega(\varphi_+\big\vert_\Omega,u)=0$ for all $\varphi_+\in\HH_1$ with $\Tr_1\varphi_+=0$.
If $\Tr_1\varphi_+=0=\Tr_1 0$, then by Condition~\eqref{cond:trace:extension} there is some $\varphi\in\HH_1$ with $\varphi\big\vert_\Omega=\varphi_+$, $\varphi\big\vert_{\CC} = 0$ and $\Tr_1\varphi=0$.

By the definition~\eqref{eqn:S} of the single layer potential,
\begin{equation*}0=\B(\varphi,\s^L\arr g)
=\B^\Omega(\varphi\big\vert_\Omega,\s^L_\Omega\arr g\big\vert_\Omega)
+\B^\CC(\varphi\big\vert_{\CC},\s^L_\Omega\arr g\big\vert_{\CC})
=\B^\Omega(\varphi_+\big\vert_\Omega,\s^L_\Omega\arr g\big\vert_\Omega)
\end{equation*}
as desired. 

Turning to the double layer potential,  if $\varphi\in\HH_1$, then by the definition~\eqref{eqn:D:+} of $\D_\Omega^\B$, formula~\eqref{eqn:D:alternate} for~$\D_\CC^\B$ and linearity of~$\B^\Omega$,
\begin{align*}
\B^\Omega(\varphi\big\vert_\Omega,\D^\B_\Omega \arr f) 
&= -\B^\Omega\bigl(\varphi\big\vert_\Omega, F\big\vert_\Omega\bigr)
+\B^\Omega\bigl(\varphi\big\vert_\Omega, \PP^L(L(\1_\Omega F))\big\vert_\Omega\bigr)
,\\
\B^\CC(\varphi\big\vert_\CC,\D^\B_\CC \arr f\big\vert_{\CC})
&=-\B^\CC\bigl(\varphi\big\vert_{\CC}, \PP^L(L(\1_\Omega F))\big\vert_{\CC}\bigr)
.\end{align*}
Subtracting and applying Condition~\eqref{cond:local},
\begin{align*}
\B^\Omega(\varphi\big\vert_\Omega,\D^\B_\Omega \arr f) 
-\B^\CC(\varphi\big\vert_\CC,\D^\B_\CC \arr f\big\vert_{\CC})
&= -\B^\Omega\bigl(\varphi\big\vert_\Omega, F\big\vert_\Omega\bigr)
+\B\bigl(\varphi, \PP^L(L(\1_\Omega F))\bigr)
.\end{align*}

By the definition~\eqref{eqn:newton} of $\PP^L$,
\begin{equation*}\B\bigl(\varphi, \PP^L(L(\1_\Omega F))\bigr) = \langle \varphi, L(\1_\Omega F)\rangle\end{equation*}
and by the definition~\eqref{dfn:L:singular} of $L(\1_\Omega F)$,
\begin{equation*}\B\bigl(\varphi, \PP^L(L(\1_\Omega F))\bigr) = \B^\Omega(\varphi\big\vert_\Omega,F\big\vert_\Omega).\end{equation*}

Thus, 
\begin{align}\label{eqn:D:solution}
\B^\Omega(\varphi\big\vert_\Omega,\D^\B_\Omega \arr f) 
-\B^\CC(\varphi\big\vert_\CC,\D^\B_\CC \arr f)
&= 0
\quad\text{for all $\varphi\in\HH_1$.}
\end{align}
In particular, as before if $\Tr_1 \varphi_+=0$ then there is some $\varphi$ with $\varphi\big\vert_\Omega=\varphi_+\big\vert_\Omega$, $\varphi\big\vert_\CC=0$ and so $\B^\Omega(\varphi\big\vert_\Omega,\D^\B_\Omega \arr f)=0$. This completes the proof.
\end{proof}

\begin{lem}\label{lem:green}
If $u\in\HH_2^\Omega$ and $(Lu)\big\vert_\Omega=0$, then 
\begin{equation*}u = -\D^\B_\Omega (\Tr_2 U) + \s^L_\Omega (\M^\B_\Omega u)\big\vert_\Omega,\quad
0 = \D^\B_\CC (\Tr_2 U) + \s^L_\CC (\M^\B_\Omega u)\big\vert_{\CC}\end{equation*}
for any $U\in\HH_2$ with $U\big\vert_\Omega=u$.
\end{lem}

\begin{proof}
By definition~\eqref{eqn:D:+} of the double layer potential,
\begin{equation*}
-\D_\Omega^\B (\Tr_2 U)
= U\big\vert_\Omega - \PP^L (L(\1_\Omega U))\big\vert_\Omega
= u - \PP^L (L(\1_\Omega u))\big\vert_\Omega
\end{equation*}
and by formula~\eqref{eqn:D:alternate}
\begin{equation*}\D_\CC^\B (\Tr_2 U) =  -\PP^L (L(\1_\Omega u))\big\vert_{\CC}.\end{equation*}
It suffices to show that $\PP^L(L(\1_\Omega u))=\s^L_\Omega(\M^\B_\Omega u)$.

Let $\varphi\in\HH_1$. By formulas~\eqref{eqn:S} and~\eqref{eqn:Neumann},
\begin{equation*}\B(\varphi,\s^L_\Omega (\M^\B_\Omega u))
=\langle \Tr_1 \varphi,\M^\B_\Omega u\rangle
=\B^\Omega(\varphi\big\vert_\Omega,u)
.\end{equation*}
By formula~\eqref{eqn:newton} for the Newton potential
and by the definition~\eqref{dfn:L:singular} of $L(\1_\Omega u)$,
\begin{equation*}
\B(\varphi, \PP^L(L(\1_\Omega u)))
= \langle \varphi, L(\1_\Omega u)\rangle
=\B^\Omega(\varphi\big\vert_\Omega,u)
.\end{equation*}
Thus, $\B(\varphi, \PP^L(L(\1_\Omega u))) = \B(\varphi,\s^L_\Omega (\M^\B_\Omega u))$ for all $\varphi\in\HH$; by coercivity of $\B$, we must have that $\PP^L(L(\1_\Omega u))=\s^L_\Omega(\M^\B_\Omega u)$. This completes the proof.
\end{proof}

Let $\B^*(\varphi,\psi)=\overline{\B(\psi,\varphi)}$ and define $\B^\Omega_*$, $\B^\CC_*$ analogously. Then $\B^*$ is a bounded and coercive operator $\HH_2\times \HH_1\mapsto\C$, and so we can define the double and single layer potentials $\D^{\B^*}_\Omega:\DD_1\mapsto \HH_1^\Omega$, $\s^{L^*}_\Omega:\NN_1\mapsto \HH_1$.

We then have the following adjoint relations.

\begin{lem}\label{lem:adjoint}
We have the adjoint relations
\begin{align}
\label{eqn:neumann:D:dual}
\langle \arr \varphi, \M_\B^{\Omega} \D^{\B}_\Omega \arr f\rangle 
&= \langle\M_{\B^*}^{\Omega} \D^{\B^*}_\Omega  \arr \varphi, \arr f\rangle
,\\
\label{eqn:dirichlet:S:dual}
\langle \arr \gamma, \Tr_2 \s^L_\Omega \arr g\rangle
&= \langle \Tr_1 \s^{L^*}_\Omega \arr \gamma, \arr g\rangle
\end{align}
for all $\arr f\in \DD_2$, $\arr \varphi\in\DD_1$, $\arr g\in\NN_2$ and $\arr\gamma\in\NN_1$.

If we let $\Tr_2^\Omega \D^{\B}_\Omega \arr f = -\Tr_2 F + \Tr_2 \PP^L(L(\1_\Omega F))$, where $F$ is as in formula~\eqref{eqn:D:+}, then $\Tr_2^\Omega \D^{\B}_\Omega \arr f $ does not depend on the choice of $F$, and we have the duality relations
\begin{align}\label{eqn:dirichlet:D:dual}
\langle \arr \gamma, \Tr_2^\Omega \D^{\B}_\Omega \arr f\rangle
&= \langle-\arr \gamma+\M_{\B^*}^{\Omega}\s^{L^*}_\Omega\arr \gamma, \arr f\rangle
.\end{align}
\end{lem}

\begin{proof}
By formula~\eqref{eqn:S},
\begin{gather*}\langle \Tr_1 \s^{L^*}_\Omega \arr \gamma, \arr g\rangle
=\B(\s^{L^*}_\Omega\arr \gamma,\s^L_\Omega\arr g\rangle,
\\
\langle \Tr_1 \s^{L}_\Omega \arr g, \arr \gamma\rangle
=\B^*(\s^{L}_\Omega\arr g,\s^{L^*}_\Omega\arr \gamma\rangle\end{gather*}
and so formula~\eqref{eqn:dirichlet:S:dual} follows by definition of~$\B^*$.

Let $\Phi\in\HH_1$ and $F\in\HH_2$ with $\Tr_1\Phi=\arr\varphi$, $\Tr_2F=\arr f$.

Then by formulas \eqref{eqn:Neumann} and~\eqref{eqn:D:+},
\begin{align*}\langle \arr \varphi, \M_\B^{\Omega} \D^{\B}_\Omega \arr f\rangle 
&=
	\B^\Omega(\Phi\big\vert_\Omega, \D^{\B}_\Omega \arr f)
=
	-\B^\Omega(\Phi\big\vert_\Omega, F\vert_\Omega)
	+\B^\Omega(\Phi\big\vert_\Omega, \PP^L(L(\1_\Omega F))\big\vert_\Omega)
\\&=
	-\overline{\B^\Omega_*(F\vert_\Omega, \Phi\big\vert_\Omega)}
	+\overline{\B^\Omega_*(\PP^L(L(\1_\Omega F))\big\vert_\Omega, \Phi\big\vert_\Omega)}
.\end{align*}
By formula~\eqref{dfn:L:singular}, 
\begin{equation*}\B^\Omega_*(\PP^L(L(\1_\Omega F))\big\vert_\Omega, \Phi\big\vert_\Omega)
= \langle \PP^L(L(\1_\Omega F)), L^*(\1_\Omega\Phi)\rangle.\end{equation*}
By formula~\eqref{eqn:newton},
\begin{equation*}\B^\Omega_*(\PP^L(L(\1_\Omega F))\big\vert_\Omega, \Phi\big\vert_\Omega)
= \B^\Omega_*( \PP^L(L(\1_\Omega F)), \PP^{L^*}(L^*(\1_\Omega\Phi))).\end{equation*}
Thus, 
\begin{align*}\langle \arr \varphi, \M_\B^{\Omega} \D^{\B}_\Omega \arr f\rangle 
&=
	-\overline{\B^\Omega_*(F\vert_\Omega, \Phi\big\vert_\Omega)}
	+\overline{\B^\Omega_*( \PP^L(L(\1_\Omega F)), \PP^{L^*}(L^*(\1_\Omega\Phi)))}
.\end{align*}
By the same argument
\begin{align*}\langle
\arr f,\M_{\B^*}^{\Omega} \D^{\B^*}_\Omega  \arr \varphi\rangle
&=
	-\overline{\B^\Omega(\Phi\vert_\Omega, F\big\vert_\Omega)}
	+\overline{\B^\Omega( \PP^{L^*}(L^*(\1_\Omega \Phi)), \PP^{L}(L(\1_\Omega F)))}
\end{align*}
and by definition of $\B^\Omega_*$ formula~\eqref{eqn:neumann:D:dual} is proven.


Finally, by definition of $\Tr_2^\Omega \D^{\B}_\Omega$,
\begin{equation*}\langle \arr \gamma, \Tr_2^\Omega \D^{\B}_\Omega \arr f\rangle
=
-\langle \arr \gamma, \Tr_2 F\rangle
+
\langle \arr \gamma, \Tr_2 \PP^L(L(\1_\Omega F)) \rangle
.\end{equation*}
By the definition~\eqref{eqn:S} of the single layer potential,
\begin{equation*}
\langle \arr \gamma, \Tr_2 \PP^L(L(\1_\Omega F)) \rangle
=
\overline{\B^*( \PP^L(L(\1_\Omega F)) ,\s^{L^*}_\Omega\arr \gamma)}
.\end{equation*}
By definition of $\B^*$ and the definition~\eqref{eqn:newton} of the Newton potential,
\begin{equation*}\overline{\B^*( \PP^L(L(\1_\Omega F)) ,\s^{L^*}_\Omega\arr \gamma)}
= {\langle \s^{L^*}_\Omega\arr \gamma, L(\1_\Omega F) \rangle}
\end{equation*}
and by the definition~\eqref{dfn:L:singular} of $L(\1_\Omega F)$,
\begin{equation*}{\langle \s^{L^*}_\Omega\arr \gamma, L(\1_\Omega F) \rangle} 
= \B^\Omega (\s^{L^*}_\Omega\arr \gamma\big\vert_\Omega, F\big\vert_\Omega).\end{equation*}
By the definition~\eqref{eqn:Neumann} of Neumann boundary values,
\begin{equation*}\overline{\B^\Omega_*(F,\s^{L^*}_\Omega\arr \gamma)} = \overline{\langle \Tr_2 F, \M^\Omega_{\B^*}(\s^{L^*}_\Omega\arr\gamma\big\vert_\Omega)\rangle}\end{equation*}
and so
\begin{equation*}\langle \arr \gamma, \Tr_2^\Omega \D^{\B}_\Omega \arr f\rangle
=
-\langle\arr\gamma,\arr f\rangle + 
{\langle \M^\Omega_{\B^*}(\s^{L^*}_\Omega\arr\gamma\big\vert_\Omega), \arr f\rangle}
\end{equation*}
for any choice of $F$. Thus $\Tr_2^\Omega \D^{\B}_\Omega $ is well-defined and formula~\eqref{eqn:dirichlet:D:dual} is valid. 
\end{proof}


\begin{lem}\label{lem:jump}
Let $\Tr_2^\Omega \D^{\B}_\Omega $ be as in Lemma~\ref{lem:adjoint}.
If $\arr f\in\DD$ and $\arr g\in\NN$, then we have the jump and continuity relations
\begin{align}
\label{eqn:D:jump}
\Tr_2^\Omega\D^\B_\Omega\arr f +\Tr_2^\CC\D^\B_\CC\arr f
	&=-\arr f
,\\
\label{eqn:S:jump}
 \M_\B^\Omega (\s^L_\Omega \arr g\big\vert_\Omega) 
+\M_\B^{\CC} (\s^L_\Omega\arr g\big\vert_{\CC})
	&=\arr g
,\\
\label{eqn:D:cts}
\M_\B^{\Omega} (\D^\B_\Omega\arr f)  - \M_\B^{\CC} (\D^\B_\CC\arr f )
	&=0
.\end{align}

If there are bounded operators $\Tr_2^\Omega:\HH_2^\Omega\mapsto\DD_2$ and $\Tr_2^\CC:\HH_2^\CC\mapsto\DD_2$ such that $\Tr_2 F = \Tr_2^\Omega (F\big\vert_\Omega)= \Tr_2^\CC (F\big\vert_\CC)$ for all $F\in\HH_2$, then in addition
\begin{align}
\label{eqn:S:cts}
\Tr_2^\Omega(\s^L_\Omega \arr g\big\vert_\Omega) -\Tr_2^\CC(\s^L_\Omega \arr g\big\vert_{\CC})
	&=0
.\end{align}
\end{lem}

The given condition formula for $\Tr_2^\Omega$, $\Tr_2^\CC$ is very natural if $\Omega\subset\R^\dmn$ is an open set, $\CC=\R^\dmn\setminus\bar\Omega$ and $\Tr_2$ denotes a trace operator restricting functions to the boundary~$\partial\Omega$.
Observe that if such operators $\Tr_2^\Omega$ and $\Tr_2^\CC$ exist, then by the definition~\eqref{eqn:D:+} of the double layer potential and by the definition of $\Tr_2^\Omega \D^\B_\Omega$ in Lemma~\ref{lem:adjoint}, $\Tr_2^\Omega (\D^\B_\Omega \arr f) = (\Tr_2^\Omega \D^\B_\Omega)\arr f$ and so there is no ambiguity of notation.

\begin{proof}[Proof of Lemma~\ref{lem:jump}]
We first observe that by the definition~\eqref{eqn:Neumann} of Neumann boundary values, if $u_+\in\HH_2^\Omega$ and $u_-\in\HH_2^\CC$ with $(Lu_+)\big\vert_\Omega=0$ and $(Lu_-)\big\vert_\CC=0$, then
\begin{equation*}
\M_\B^\Omega u_+ 
+\M_\B^{\CC} u_-=\arr \psi
\text{ if and only if }
\langle\Tr_1\varphi,\arr\psi\rangle = \B^\Omega(\varphi\big\vert_\Omega,u_+)
+\B^\CC(\varphi\big\vert_{\CC},u_-)
\end{equation*}
for all $\varphi\in\HH_1$.

The continuity relation \eqref{eqn:D:cts} follows from formula~\eqref{eqn:D:solution}. 

The jump relation~\eqref{eqn:D:jump} follows from the definition of $\Tr_2^\Omega\D^\B_\Omega$ and by using formula~\eqref{eqn:D:alternate:extensions} to rewrite $\Tr_2^\CC\D^\B_\CC$.

The jump relation \eqref{eqn:S:jump} follows from the definition~\eqref{eqn:S} of the single layer potential.

The continuity relation~\eqref{eqn:S:cts} follows because $\s^L_\Omega\arr g\in \HH_2$ and by the definition of $\Tr_2^\Omega$, $\Tr_2^\CC$.
\end{proof}

%

%
%

\section{Layer potentials and boundary value problems}
\label{sec:invertible}



If $\HH_2^\Omega$ and $\DD_2$ are as in Section~\ref{sec:example}, then by the Lax-Milgram lemma there is a unique solution to the Dirichlet aed and Neumann boundary value problems
\begin{equation*}\left\{\begin{aligned}
(Lu)\big\vert_\Omega &= 0,\\
\Tr^\Omega_2 u &= \arr f,\\
\doublebar{u}_{\HH_2^\Omega} &\leq C \doublebar{\arr f}_{\DD_2},
\end{aligned}\right.\qquad
\left\{\begin{aligned}
(Lu)\big\vert_\Omega &= 0,\\
\M_\B^\Omega u &= \arr g,\\
\doublebar{u}_{\HH^\Omega_2} &\leq C \doublebar{\arr g}_{\NN_2}
.\end{aligned}\right.
\end{equation*}
We routinely wish to establish existence and uniqueness to the Dirichlet and Neumann boundary value problems
\begin{equation*}
(D)^{\widehat L}_\XX\left\{\begin{aligned}
(\widehat Lu)\big\vert_\Omega &= 0,\\
\WTr_\XX^\Omega u &= \arr f,\\
\doublebar{u}_{\XX^\Omega} &\leq C \doublebar{\arr f}_{\DD_\XX},
\end{aligned}\right.\qquad
(N)^\B_\XX\left\{\begin{aligned}
(\widehat Lu)\big\vert_\Omega &= 0,\\
\MM_\B^\Omega u &= \arr g,\\
\doublebar{u}_{\XX^\Omega} &\leq C \doublebar{\arr g}_{\NN_\XX}
\end{aligned}\right.
\end{equation*}
for some constant $C$ and some other solution space $\XX$ and spaces of Dirichlet and Neumann boundary data $\DD_\XX$ and $\NN_\XX$. For example, if $L$ is a second-order differential operator, then as in \cite{JerK81B,KenP93,KenR09,DinPR13p} we might wish to establish well-posedness with $\DD_\XX=\dot W_1^p(\partial\Omega)$, $\NN_\XX=L^p(\partial\Omega)$ and $\XX=\{u:\widetilde N(\nabla u)\in L^p(\partial\Omega)\}$, where $\widetilde N$ is the nontangential maximal function introduced in \cite{KenP93}.

The classic method of layer potentials states that if layer potentials, originally defined as bounded operators $\D^\B_\Omega:\DD_2\mapsto\HH_2$ and $\s^L_\Omega:\NN_2\mapsto\HH_2$, may be extended to operators $\D^\B_\Omega:\DD_\XX\mapsto\XX$ and $\s^L_\Omega:\NN_\XX\mapsto\XX$, and if certain of the properties of layer potentials of Section~\ref{sec:properties} are preserved by that extension, then well posedness of boundary value problems are equivalent to certain invertibility properties of layer potentials.

In this section we will make this notion precise. 

As in Sections~\ref{sec:dfn}, \ref{sec:D:S} and~\ref{sec:properties}, we will work with layer potentials and function spaces in a very abstract setting. 

\subsection{From invertibility to well posedness}
\label{sec:invertible:well-posed}

In this section we will need the following objects.
\begin{itemize}
\item Quasi-Banach spaces $\XX^\Omega$, $\DD_\XX$ and $\NN_\XX$.
\item A linear operator $u\mapsto(\widehat L u)\big\vert_\Omega$ acting on~$\XX^\Omega$. 
\item Linear operators $\WTr_\XX^\Omega:\{u\in\XX^\Omega:(\widehat L u)\big\vert_\Omega =0\}\mapsto\DD_\XX$ and $\MM_\B^\Omega:\{u\in\XX^\Omega:(\widehat L u)\big\vert_\Omega =0\}\mapsto\NN_\XX$.
\item Linear operators $\widehat\D^\B_\Omega:\DD_\XX\mapsto \XX^\Omega$ and $\widehat\s^L_\Omega:\NN_\XX\mapsto \XX^\Omega$.
\end{itemize}
\begin{rmk}
Recall that $\s^L_\Omega=\s^L_\CC$ is defined in terms of a ``global'' Hilbert space $\HH_2$. If $\XX^\Omega=\HH^\Omega$, then $\widehat \s^L_\Omega\arr g = \s^L_\Omega\arr g\big\vert_\Omega$. In the general case, we do not assume the existence of a global quasi-Banach space $\XX$ whose restrictions to $\Omega$ lie in~$\XX^\Omega$, and thus we will let $\widehat \s^L_\Omega\arr g$ be an element of $\XX^\Omega$ without assuming a global extension.
\end{rmk}


In applications it is often useful to define $\Tr^\Omega$, $\M_\B^\Omega$, $L$, $\D^\B_\Omega$ and $\s^L_\Omega$ in terms of some Hilbert spaces $\HH_j$, $\HH_j^\Omega$ and to extend these operators to operators with domain or range $\XX^\Omega$ by density or some other means. See, for example, \cite{BarHM17pC}. We will not assume that the operators $\WTr_\XX^\Omega$, $\MM_\B^\Omega$, $\widehat L$, $\widehat\D^\B_\Omega$ and $\widehat\s^L_\Omega$ arise by density; we will merely require that they satisfy certain properties similar to those established in Section~\ref{sec:properties}. 


Specifically, we will often use the following conditions; observe that if $\XX^\Omega=\HH_2^\Omega$ for some $\HH_2^\Omega$ as in Section~\ref{sec:dfn}, these properties are valid.
\begin{description}
\item[\hypertarget{ConditionT}{\condTname}] 
$\WTr_\XX^\Omega$ is a bounded operator $\{u\in\XX^\Omega:(\widehat L u)\big\vert_\Omega =0\}\mapsto \DD_\XX$.
\item[\hypertarget{ConditionM}{\condMname}] 
$\MM_\B^\Omega$ is a bounded operator $\{u\in\XX^\Omega:(\widehat L u)\big\vert_\Omega =0\}\mapsto \NN_\XX$.
\item[\hypertarget{ConditionS}{\condSname}] The single layer potential $\widehat\s^L_\Omega$ is bounded $\NN_\XX\mapsto \XX^\Omega$, and if $\arr g\in \NN_\XX$ then $(\widehat L (\widehat\s^L_\Omega\arr g))\big\vert_\Omega=0$.
\item[\hypertarget{ConditionD}{\condDname}] The double layer potential $\widehat\D^\B_\Omega$ is bounded $\DD_\XX\mapsto\XX^\Omega$, and if $\arr f\in \DD_\XX$ then $(\widehat L (\widehat\D^\B_\Omega\arr f))\big\vert_\Omega=0$.
\item[\hypertarget{ConditionG}{\condGname}] If $u \in \XX^\Omega$ and $(\widehat L u)\big\vert_\Omega =0$, then we have the Green's formula
\begin{equation*}u = -\widehat\D^\B_\Omega (\WTr_\XX^\Omega u) + \widehat\s^L_\Omega (\MM_\B^\Omega u).\end{equation*}
\end{description}

We remark that the linear operator $u\mapsto(\widehat L u)\big\vert_\Omega$ is used only to characterize the subspace $\XX^\Omega_{ L}=\{u\in\XX^\Omega:(\widehat L u)\big\vert_\Omega =0\}$. We could work directly with $\XX^\Omega_{ L}$; however, we have chosen to use the more cumbersome notation $\{u\in\XX^\Omega:(\widehat L u)\big\vert_\Omega =0\}$ to emphasize that the following arguments, presented here only in terms of linear operators, are intended to be used in the context of boundary value problems for differential equations.



The following theorem is straightforward to prove and is the core of the classic method of layer potentials.
\begin{thm}\label{thm:surjective:existence}
Let $\XX$, $\DD_\XX$, $\NN_\XX$, $\widehat\D^\B_\Omega$, $\widehat\s^L_\Omega$, $\WTr_\XX^\Omega$ and $\MM_\B^\Omega$ be quasi-Banach spaces and linear operators with domains and ranges as above.

Suppose that Conditions~\condT\ and \condS\ are valid, and that $\WTr_\XX^\Omega \widehat\s^L_\Omega: \NN_\XX \mapsto \DD_\XX$ is surjective. 
Then for every $\arr f\in \DD_\XX$, there is some $u$ such that
\begin{equation}\label{eqn:Dirichlet:weak}
(\widehat L u)\big\vert_\Omega = 0, \quad \WTr_\XX^\Omega u = \arr f, \quad u\in\XX^\Omega.\end{equation}
Suppose in addition $\WTr_\XX^\Omega \widehat\s^L_\Omega: \NN_\XX \mapsto \DD_\XX$ has a bounded right inverse, that is, there is a constant $C_0$ such that if $\arr f\in\DD_\XX$, then there is some preimage $\arr g$ of $\arr f$ with $\doublebar{\arr g}_{\NN_\XX}\leq C_0\doublebar{\arr f}_{\DD_\XX}$. Then there is some constant $C_1$ depending on $C_0$ and the implicit constants in Conditions~\condT\ and~\condS\ such that if $\arr f\in\DD_\XX$, then there is some $u\in\XX^\Omega$ such that
\begin{equation}
\label{eqn:Dirichlet:strong}
(\widehat L u)\big\vert_\Omega = 0, \quad \WTr_\XX^\Omega u = \arr f, \quad \doublebar{u}_{\XX^\Omega}\leq C_1\doublebar{\arr f}_{\DD_\XX}.\end{equation}

Suppose that Conditions~\condM\ and \condD\ are valid, and that $\MM_\B^\Omega\widehat\D^\B_\Omega: \DD_\XX \mapsto \NN_\XX$ is surjective. 
Then for every $\arr g\in \NN_\XX$, there is some $u$ such that
\begin{equation}\label{eqn:Neumann:weak}(\widehat L u)\big\vert_\Omega = 0, \quad \MM_\B^\Omega u = \arr g, \quad u\in\XX^\Omega.\end{equation}
If $\MM_\B^\Omega\widehat\D^\B_\Omega: \DD_\XX \mapsto \NN_\XX$ has a bounded right inverse, then there is some constant $C_1$ depending on the bound on that inverse and the implicit constants in Conditions~\condM\ and~\condD\ such that if $\arr g\in\NN_\XX$, then there is some $u\in\XX^\Omega$ such that
\begin{equation}\label{eqn:Neumann:strong}(\widehat L u)\big\vert_\Omega = 0, \quad \MM_\B^\Omega u = \arr g, \quad \doublebar{u}_{\XX^\Omega}\leq C_1\doublebar{\arr g}_{\NN_\XX}.\end{equation}

\end{thm}

Thus, surjectivity of layer potentials implies of solutions to for boundary value problems. 

We may also show that injectivity of layer potentials implies uniqueness of solutions to boundary value problems. This argument appeared first in \cite{BarM16A} and is the converse to an argument of \cite{Ver84}.

\begin{thm}\label{thm:injective:unique}
Let $\XX$, $\DD_\XX$, $\NN_\XX$, $\widehat\D^\B_\Omega$, $\widehat\s^L_\Omega$, $\WTr_\XX^\Omega$ and $\MM_\B^\Omega$ be quasi-Banach spaces and linear operators with domains and ranges as above. Suppose that Conditions~\condT, \condM, \condD, \condS\ and~\condG\ are all valid.

Suppose that the operator 
$\WTr_\XX^\Omega \widehat\s^L_\Omega: \NN_\XX \mapsto \DD_\XX$ is one-to-one. Then for each $\arr f\in\DD_\XX$, there is at most one solution $u$ to the Dirichlet problem
\begin{equation*}(\widehat L u)\big\vert_\Omega = 0, \quad \WTr_\XX^\Omega u = \arr f, \quad u\in\XX^\Omega.\end{equation*}
If $\WTr_\XX^\Omega \widehat\s^L_\Omega: \NN_\XX \mapsto \DD_\XX$ has a bounded left inverse, that is, there is a constant $C_0$ such that the estimate $\doublebar{\arr g}_{\NN_\XX}\leq C_0 \doublebar{\WTr_\XX^\Omega \widehat\s^L_\Omega \arr g}_{\DD_\XX}$ is valid, then there is some constant $C_1$ such that every $u\in\XX^\Omega$ with $(\widehat L u)\big\vert_\Omega=0$ satisfies $\doublebar{u}_{\XX^\Omega}\leq C_1\doublebar{\Tr_\XX^\Omega u}_{\DD_\XX}$ (that is, if $u$ satisfies the Dirichlet problem~\eqref{eqn:Dirichlet:weak} then $u$ must satisfy the Dirichlet problem~\eqref{eqn:Dirichlet:strong}).

Similarly, if the operator $\MM_\B^\Omega \widehat\D^\B_\Omega:\DD_\XX \mapsto \NN_\XX$ is one-to-one, then for each $\arr g\in\NN_\XX$, there is at most one solution $u$ to the Neumann problem
\begin{equation*}(\widehat L u)\big\vert_\Omega = 0, \quad \MM_\B^\Omega u = \arr g, \quad u\in\XX^\Omega.\end{equation*}
If $\MM_\B^\Omega \widehat\D^\B_\Omega:\DD_\XX \mapsto \NN_\XX$ has a bounded left inverse, then there is some constant $C_1$ such that every $u\in\XX^\Omega$ with $(\widehat L u)\big\vert_\Omega=0$ satisfies $\doublebar{u}_{\XX^\Omega}\leq C_1\doublebar{\MM_\B^\Omega u}_{\DD_\XX}$. 
\end{thm}

\begin{proof} 
We present the proof only for the Neumann problem; the argument for the Dirichlet problem is similar.

Suppose that $u$, $v\in\XX^\Omega$ with $(\widehat Lu)\big\vert_\Omega=(\widehat Lv)\big\vert_\Omega=0$ in~$\Omega$ and $\MM_\B^\Omega u=\arr g=\MM_\B^\Omega v$. By Condition~\condG,
\begin{align*}
u
= -\widehat\D^\B_\Omega (\WTr_\XX^\Omega u) + \widehat\s^L_\Omega (\MM_\B^\Omega u)
&= -\widehat\D^\B_\Omega (\WTr_\XX^\Omega u) + \widehat\s^L_\Omega \arr g
,\\
v= -\widehat\D^\B_\Omega (\WTr_\XX^\Omega v) + \widehat\s^L_\Omega (\MM_\B^\Omega v)
&= -\widehat\D^\B_\Omega (\WTr_\XX^\Omega v) + \widehat\s^L_\Omega \arr g
.\end{align*}
In particular, $\MM_\B^\Omega \widehat\D^\B_\Omega (\WTr_\XX^\Omega u) = \MM_\B^\Omega \widehat\D^\B_\Omega (\WTr_\XX^\Omega v)$. If $\MM_\B^\Omega \widehat\D^\B_\Omega$ is one-to-one, then $\WTr_\XX^\Omega u = \WTr_\XX^\Omega v$. Another application of Condition~\condG\ yields that $u=v$.

Now, suppose that we have the estimate $\doublebar{\arr f}_{\DD_\XX}\leq C \doublebar{\MM_\B^\Omega \widehat\D^\B_\Omega \arr f}_{\NN_\XX}$. (This implies injectivity of $\MM_\B^\Omega \widehat\D^\B_\Omega$.) Let $u\in {\XX^\Omega}$ with $(\widehat L u)\big\vert_\Omega=0$; we want to show that $\doublebar{u}_{\XX^\Omega}\leq C \doublebar{\MM_\B^\Omega u}_{\DD_\XX}$.

By Condition~\condG, and because $\XX^\Omega$ is a quasi-Banach space,
\begin{equation*}\doublebar{u}_{\XX^\Omega}\leq C\doublebar{\widehat\D^\B_\Omega (\WTr_\XX^\Omega  u)}_{\XX^\Omega}  + C\doublebar{\widehat\s^L_\Omega (\MM_\B^\Omega u)}_{\XX^\Omega}.\end{equation*}
By Conditions~\condD\ and~\condS,
\begin{equation*}\doublebar{u}_\XX\leq C\doublebar{\WTr_\XX^\Omega  u}_{\DD_\XX}  + C\doublebar{\MM_\B^\Omega u}_{\NN_\XX}.\end{equation*}
Applying our estimate on $\MM_\B^\Omega \widehat\D^\B_\Omega$, we see that
\begin{equation*}\doublebar{u}_\XX\leq C\doublebar{\MM_\B^\Omega \widehat\D^\B_\Omega\WTr_\XX^\Omega  u}_{\NN_\XX}  + C\doublebar{\MM_\B^\Omega u}_{\NN_\XX}.\end{equation*}
By Condition~\condG,
$\widehat\D^\B_\Omega(\WTr_\XX^\Omega u) =\widehat\s^L_\Omega(\MM_\B^\Omega u) - u $, and so
\begin{equation*}\doublebar{u}_\XX\leq C\doublebar{\MM_\B^\Omega \widehat \s^L_\Omega\MM_\B^\Omega u}_{\NN_\XX}  + C\doublebar{\MM_\B^\Omega u}_{\NN_\XX}.\end{equation*}
Another application of Condition~\condS\ and of Condition~\condM\  completes the proof. \end{proof}

\subsection{From well posedness to invertibility}
\label{sec:well-posed:invertible}

We are now interested in the converse results. That is, we have shown that results for layer potentials imply results for boundary value problems; we would like to show that results for boundary value problems imply results for layer potentials.

Notice that the above results were built on the Green's formula~\condG. The converse results will be built on jump relations, as in Lemma~\ref{lem:jump}. Recall that jump relations treat the interplay between layer potentials in a domain and in its complement; thus we will need to impose conditions in both domains. 

In this section we will need the following spaces and operators.
\begin{itemize}
\item Quasi-Banach spaces $\XX^\UU$, $\XX^\VV$, $\DD_\XX$ and $\NN_\XX$.
\item Linear operators $u\mapsto(\widehat L u)\big\vert_\UU$ and $u\mapsto(\widehat L u)\big\vert_\VV$ acting on~$\XX^\UU$ and~$\XX^\VV$. 
\item Linear operators $\WTr_\XX^\UU $, $\MM_\B^\UU$, $\WTr_\XX^\VV $, and $\MM_\B^\VV$ acting on $\{u\in\XX^\UU:(\widehat L u)\big\vert_\UU =0\}$ or $\{u\in\XX^\VV:(\widehat L u)\big\vert_\VV =0\}$.
\item Linear operators $\widehat\D^\B_\UU$, $\widehat\D^\B_\VV$ acting on $\DD_\XX$ and $\widehat\s^L_\UU$, $\widehat\s^L_\VV$ acting on $\NN_\XX$.
\end{itemize}
In the applications $\UU$ is an open set in $\R^\dmn$ or in a smooth manifold, and $\VV=\R^\dmn\setminus\overline\UU$ is the interior of its complement. The space $\XX^\VV$ is then a space of functions defined in~$\VV$ and is thus a different space from $\XX^\UU$. However, we emphasize that we have defined only one space $\DD_\XX$ of Dirichlet boundary values and one space $\NN_\XX$ of Neumann boundary values; that is, the traces from both sides of the boundary must lie in the same spaces.

We will often use the following conditions.
\begin{description}
\item[\hypertarget{ConditionTT}{\condTTname}, \hypertarget{ConditionMM}{\condMMname}, \hypertarget{ConditionSS}{\condSSname}, \hypertarget{ConditionDD}{\condDDname}, \hypertarget{ConditionGG}{\condGGname}] Condition \condT, \condM, \condS, \condD, or \condG\ holds for both $\Omega=\UU$ and $\Omega=\VV$.
\item[\hypertarget{ConditionJScts}{\condJSctsname}] If $\arr g\in\NN_\XX$, then we have the continuity relation
\begin{align*}
\WTr_\XX^\UU (\widehat\s^L_\UU \arr g) -\WTr_\XX^\VV (\widehat\s^L_\VV \arr g)
	&=0
.\end{align*}
\item[\hypertarget{ConditionJDcts}{\condJDctsname}] If $\arr f\in\DD_\XX$, then we have the continuity relation
\begin{align*}
\MM_\B^\UU (\widehat\D^\B_\UU\arr f)  - \MM_\B^\VV (\widehat\D^\B_\VV\arr f )
	&=0
.\end{align*}
\item[\hypertarget{ConditionJSjump}{\condJSjumpname}] If $\arr g\in\NN_\XX$, then we have the jump relation
\begin{align*}
 \MM_\B^\UU (\widehat\s^L_\UU ) 
+\MM_\B^\VV (\widehat\s^L_\VV\arr g)
	&=\arr g
.\end{align*}
\item[\hypertarget{ConditionJDjump}{\condJDjumpname}] If $\arr f\in\DD_\XX$, then we have the jump relation
\begin{align*}
\WTr_\XX^\UU (\widehat\D^\B_\UU\arr f)  +\WTr_\XX^\VV (\widehat\D^\B_\VV\arr f)
	&=-\arr f
.\end{align*}
\end{description}

We now move from well posedness of boundary value problems to invertibility of layer potentials.

The following theorem uses an argument of Verchota from \cite{Ver84}.

\begin{thm}\label{thm:unique:injective}
Assume that Conditions~\condMM, 
\condSS, \condJScts, and \condJSjump\  are valid.
Suppose that, for any $\arr f\in\DD_\XX$, there is at most one solution $u_+$ or $u_-$ to each of the two Dirichlet problems
\begin{gather*}
(\widehat L u_+)\big\vert_\UU = 0, \quad \WTr_\XX^\UU  u_+ = \arr f, \quad u_+\in\XX^\UU
,\\
(\widehat L u_-)\big\vert_\VV = 0, \quad \WTr_\XX^\VV  u_- = \arr f, \quad u_-\in\XX^\VV
.\end{gather*}
Then $\Tr_\XX^{\UU}\widehat\s^L_\UU:\NN_\XX\mapsto\DD_\XX$ is one-to-one.

If in addition there is a constant $C_0$ such that every $u_+\in\XX^\UU$ and $u_-\in\XX^\VV$ with $(\widehat L u_+)\big\vert_\UU =0$ and $(\widehat L u_+)\big\vert_\VV = 0$ satisfies
\begin{equation*}\doublebar{u_+}_{\XX^\UU}\leq C_0\doublebar{\WTr_\XX^\UU  u}_{\DD_\XX},
\quad
\doublebar{u_-}_{\XX^\VV}\leq C_0\doublebar{\WTr_\XX^\VV  u}_{\DD_\XX},
\end{equation*}
then there is a constant $C_1$ such that the bound $\doublebar{\arr g}_{\NN_\XX} \leq C_1 \doublebar{\Tr_\XX^{\UU}\widehat\s^L_\UU\arr g}_{\DD_\XX}$ is valid for all $\arr g\in\NN_\XX$.

Similarly, assume that Conditions~
\condTT, \condDD, \condJDcts, and \condJDjump\  are valid. Suppose that for any $\arr g\in\NN_\XX$, there is at most one solution $u_\pm$ to each of the two Neumann problems
\begin{gather*}
(\widehat L u_+)\big\vert_\UU = 0, \quad \MM_\B^\UU  u_+ = \arr g, \quad u_+\in\XX^\UU
,\\
(\widehat L u_-)\big\vert_\VV = 0, \quad \MM_\B^\VV  u_- = \arr g, \quad u_-\in\XX^\VV
.\end{gather*}
Then $\MM_\B^\UU\widehat\D^\B_\UU:\DD_\XX\mapsto\NN_\XX$ is one-to-one.

If there is a constant $C_0$ such that every $u_+\in\XX^\UU$ and $u_-\in\XX^\VV$ with $(\widehat L u_+)\big\vert_\UU =0$ and $(\widehat L u_+)\big\vert_\VV = 0$ satisfies
\begin{equation*}\doublebar{u_+}_{\XX^\UU}\leq C_0\doublebar{\MM_\B^\UU u}_{\DD_\XX},
\quad
\doublebar{u_-}_{\XX^\VV}\leq C_0\doublebar{\MM_\B^\VV u}_{\DD_\XX},
\end{equation*}
then there is a constant $C_1$ such that the bound $\doublebar{\arr f}_{\DD_\XX} \leq C_1 \doublebar{\MM_\B^\UU\widehat\D^\B_\UU\arr f}_{\NN_\XX}$ is valid for all $\arr f\in\DD_\XX$. 
\end{thm}

\begin{proof}
As in the proof of Theorem~\ref{thm:injective:unique}, we will consider only the relationship between the Neumann problem and the double layer potential.

Let $\arr f$, $\arr h\in\DD_\XX$. By Condition~\condDD, $u_+=\widehat\D^\B_\UU\arr f\in\XX^\UU$ and $v_+=\widehat\D^\B_\UU \arr h\in\XX^\UU$. If $\MM_\B^\UU \widehat\D^\B_\UU \arr f = \MM_\B^\UU \widehat\D^\B_\UU \arr h$, then $\MM_\B^\UU u_+=\MM_\B^\UU v_+$. Because there is at most one solution to the Neumann problem, we must have that $u_+=v_+$, and in particular $\WTr_\XX^\UU  \widehat\D^\B_\UU \arr f = \WTr_\XX^\UU  \widehat\D^\B_\UU \arr h$.

By Condition~\condJDcts, we have that $\MM_\B^\VV \widehat\D^\B_\VV \arr f = \MM_\B^\VV \widehat\D^\B_\VV \arr h$. By Condition~\condDD\ and uniqueness of solutions to the $\VV$-Neumann problem, $\WTr_\XX^\VV  \widehat\D^\B_\VV \arr f = \WTr_\XX^\VV  \widehat\D^\B_\VV \arr h$. By \condJDjump, we have that
\begin{equation*}\arr f 
= \WTr_\XX^\UU  \widehat\D^\B_\UU \arr f + \WTr_\XX^\VV  \widehat\D^\B_\VV \arr f
= \WTr_\XX^\UU  \widehat\D^\B_\UU \arr h + \WTr_\XX^\VV  \widehat\D^\B_\VV \arr h
=\arr h\end{equation*}
and so $\MM_\B^\UU \widehat\D^\B_\UU$ is one-to-one.

Now assume the stronger condition, that is, that $C_0<\infty$. Because $\DD_\XX$ is a quasi-Banach space, if $\arr f\in\DD_\XX$ then by Condition~\condJDjump,
\begin{equation*}\doublebar{\arr f}_{\DD_\XX} \leq C\doublebar{\WTr_\XX^\UU  \widehat\D^\B_\UU \arr f}_{\DD_\XX}+\doublebar{\WTr_\XX^\VV  \widehat\D^\B_\VV \arr f}_{\DD_\XX}.\end{equation*}
By Condition~\condDD, $\widehat\D^\B_\UU \arr f\in\XX^\UU$ with $(\widehat L(\widehat\D^\B_\UU \arr f))\big\vert_\UU=0$. Thus by Condition~\condTT, $\doublebar{\WTr_\XX^\UU  \widehat\D^\B_\UU \arr f}_{\NN_\XX}\leq C\doublebar{\widehat\D^\B_\UU \arr f}_{\XX^\UU}$.
Thus,
\begin{equation*}\doublebar{\arr f}_{\DD_\XX} \leq C\doublebar{\widehat\D^\B_\UU \arr f}_{\XX^\UU}+\doublebar{\widehat\D^\B_\VV \arr f}_{\XX^\VV}.\end{equation*}
By definition of~$C_0$, 
\begin{equation*}\doublebar{\widehat\D^\B_\UU \arr f}_{\XX^\UU}\leq C_0\doublebar{\MM_\B^\UU\widehat\D^\B_\UU \arr f}_{\NN_\XX}\quad\text{and}\quad\doublebar{\widehat\D^\B_\VV \arr f}_{\XX^\VV}\leq C_0\doublebar{\MM_\B^\VV\widehat\D^\B_\VV \arr f}_{\NN_\XX}.\end{equation*} By Condition~\condJDcts, $\MM_\B^\VV\widehat\D^\B_\VV \arr f=\MM_\B^\UU\widehat\D^\B_\UU \arr f$ and so
\begin{equation*}\doublebar{\arr f}_{\DD_\XX} \leq 2CC_0\doublebar{\MM_\B^\UU\widehat\D^\B_\UU \arr f}_{\NN_\XX}\end{equation*}
as desired.
\end{proof}


Finally, we consider the relationship between existence and surjectivity. The following argument appeared first in \cite{BarM13}.

\begin{thm} \label{thm:existence:surjective}
Assume that Conditions~\condMM, \condGG,
\condJScts, and \condJDjump\  are valid.  
Suppose that, for any $\arr f\in\DD_\XX$, there is at least one pair of solutions $u_\pm$ to the pair of  Dirichlet problems
\begin{equation}\label{eqn:Dirichlet:ES}
(\widehat L u_+)\big\vert_\UU = (\widehat L u_-)\big\vert_\VV = 0, \>\>\> \WTr_\XX^\UU  u_+ = \WTr_\XX^\VV  u_- = \arr f, \>\>\> u_+\in\XX^\UU
, \>\>\> u_-\in\XX^\VV
.\end{equation}
Then $\Tr_\XX^{\UU}\widehat\s^L_\UU:\NN_\XX\mapsto\DD_\XX$ is onto.

Suppose that there is some $C_0<\infty$ such that,  if $\arr f\in\DD_\XX$, there is some pair of solutions $u^\pm$ to the problem~\eqref{eqn:Dirichlet:ES} with
\begin{equation*}\doublebar{u_+}_{\XX^\UU}\leq C_0\doublebar{\arr f}_{\DD_\XX},
\quad
\doublebar{u_-}_{\XX^\VV}\leq C_0\doublebar{\arr f}_{\DD_\XX}.
\end{equation*}
Then there is a constant $C_1$ such that for any $\arr f\in\DD_\XX$, there is a $\arr g\in\NN_\XX$ such that ${\Tr_\XX^{\UU}\widehat\s^L_\UU\arr g}=\arr f$ and
$\doublebar{\arr g}_{\NN_\XX} \leq C_1 \doublebar{\arr f}_{\DD_\XX}$.

Similarly, assume that Conditions~\condTT, \condGG,
\condJDcts, and \condJSjump\  are valid. Suppose that for any $\arr g\in\NN_\XX$, there is at least one pair of solutions $u_\pm$ to the pair of Neumann problems
\begin{equation}\label{eqn:Neumann:ES}
(\widehat L u_+)\big\vert_\UU = (\widehat L u_-)\big\vert_\VV = 0, \>\>\> \MM_\B^\UU u_+ = \MM_\B^\VV u_- = \arr g, \>\>\> u_+\in\XX^\UU
, \>\>\> u_-\in\XX^\VV
.\end{equation}
Then $\MM_\B^\UU\widehat\D^\B_\UU:\DD_\XX\mapsto\NN_\XX$ is onto.

Suppose that there is some $C_0<\infty$ such that, if $\arr g\in\NN_\XX$, there is some pair of solutions $u^\pm$ to the problem~\eqref{eqn:Neumann:ES} with
\begin{equation}\label{eqn:Neumann:bound}\doublebar{u_+}_{\XX^\UU}\leq C_0\doublebar{\arr g}_{\NN_\XX},
\quad
\doublebar{u_-}_{\XX^\VV}\leq C_0\doublebar{\arr g}_{\NN_\XX}.
\end{equation}
Then there is a constant $C_1$ such that for any $\arr g\in\NN_\XX$, there is an $\arr f\in\DD_\XX$ such that ${\MM_\B^\UU\widehat\D^\B_\UU\arr f}=\arr g$ and
$\doublebar{\arr f}_{\DD_\XX} \leq C_1 \doublebar{\arr g}_{\NN_\XX}$.
\end{thm}

\begin{proof}
As usual we present the proof for the Neumann problem.
Choose some $\arr g\in\NN_\XX$ and let $u_+$ and $u_-$ be the solutions to the problem~\eqref{eqn:Neumann:ES} assumed to exist. (If $C_0<\infty$ we further require that the bound~\eqref{eqn:Neumann:bound} be valid.)

By Condition~\condTT, $\arr f_+=\WTr_\XX^\UU  u_+$ and $\arr f_-=\WTr_\XX^\VV  u_-$ exist and lie in~$\DD_\XX$. By Condition~\condGG,
\begin{align*}
2\arr g &= \MM_\B^\UU u_+ + \MM_\B^\VV u_-
\\&= \MM_\B^\UU (-\widehat\D^\B_\UU \arr f_+ + \widehat\s^L_\UU \arr g) + \MM_\B^\VV (-\widehat\D^\B_\VV \arr f_- + \widehat\s^L_\VV \arr g)
.\end{align*}
By Conditions~\condJDcts\ and \condJSjump\ and linearity of the operators $\MM_\B^\UU$, $\MM_\B^\VV$, we have that
\begin{align*}
2\arr g 
&=  
-\MM_\B^\UU\widehat\D^\B_\UU \arr f_+ 
+ \MM_\B^\UU\widehat\s^L_\UU \arr g 
-\MM_\B^\UU\widehat\D^\B_\UU \arr f_- 
+\arr g- \MM_\B^\UU\widehat\s^L_\UU \arr g
\\&=
\arr g -\MM_\B^\UU \widehat\D^\B_\UU (\arr f_+ + \arr f_-)
.\end{align*}
Thus, $\MM_\B^\UU \widehat\D^\B_\UU$ is surjective. If $C_0<\infty$, then because $\DD_\XX$ is a quasi-Banach space and by Condition~\condTT,
\begin{equation*}\doublebar{\arr f_+ + \arr f_-}_{\DD_\XX}
\leq C C_0\doublebar{\arr g}_{\NN_\XX}
\end{equation*}
as desired.
\end{proof}


\newcommand{\etalchar}[1]{$^{#1}$}
\providecommand{\bysame}{\leavevmode\hbox to3em{\hrulefill}\thinspace}
\providecommand{\MR}{\relax\ifhmode\unskip\space\fi MR }
\providecommand{\MRhref}[2]{%
  \href{http://www.ams.org/mathscinet-getitem?mr=#1}{#2}
}
\providecommand{\href}[2]{#2}

\end{document}